\documentclass{article}
\usepackage{amsfonts,amssymb,mathrsfs,amscd,amsmath,amsthm}
\usepackage{verbatim}

\usepackage[mathscr]{eucal}
\usepackage[dvips]{graphics}
\usepackage[dvips]{graphicx}
\usepackage{floatflt}
\usepackage[T1]{fontenc}
\usepackage{graphicx}
\usepackage{indentfirst}

\usepackage{url}

\usepackage{mathtools}

\def\thtext#1{
  \catcode`@=11
  \gdef\@thmcountersep{. #1}
  \catcode`@=12
}

\def\threst{
  \catcode`@=11
  \gdef\@thmcountersep{.}
  \catcode`@=12
}

\theoremstyle{plain}
\newtheorem{thm}{Theorem}[section]
\newtheorem{lem}[thm]{Lemma}
\newtheorem{cor}[thm]{Corollary}
\newtheorem{ass}[thm]{Assertion}

\theoremstyle{definition}
\newtheorem{dfn}[thm]{Definition}
\newtheorem{rk}[thm]{Remark}
\newtheorem{constr}[thm]{Construction}
\newtheorem{examp}[thm]{Example}

\catcode`@=11
\def\.{.\spacefactor\@m}
\catcode`@=12

\def\R{\mathbb R}

\def\a{\alpha}

\def\e{\varepsilon}
\def\dl{\delta}
\def\D{\Delta}
\def\g{\gamma}

\def\l{\lambda}

\def\s{\sigma}

\def\t{\tau}

\def\0{\emptyset}
\def\:{\colon}
\def\<{\langle}
\def\>{\rangle}

\def\rom#1{\emph{#1}}
\def\({\rom(}
\def\){\rom)}
\def\sm{\setminus}
\def\ss{\subset}

\def\x{\times}

\def\diam{\operatorname{diam}}

\def\dis{\operatorname{dis}}

\def\GH{\operatorname{\mathcal{G\!H}}}

\def\opt{{\operatorname{opt}}}

\def\cB{{\cal B}}
\def\cC{{\cal C}}
\def\cD{{\cal D}}

\def\cH{{\cal H}}

\def\cM{{\cal M}}

\def\cP{{\cal P}}
\def\cR{{\cal R}}

\def\gC{{\mathfrak C}}

\begin{document}
 \title{Extendability of Metric Segments in Gromov--Hausdorff Distance}
\author{S.I.~Borzov, A.O.~Ivanov, A.A.~Tuzhilin}
\date{}
\maketitle

\begin{abstract}
In this paper geometry of Gromov--Hausdorff distance on the class of all metric spaces considered up to an isometry is investigated. For this class continuous curves and their lengths are defined, and it is shown that the Gromov--Hausdorff distance is intrinsic. Besides, metric segments are considered, i.e., the classes of points lying between two given ones, and an extension problem of such segments beyond their end-points is considered.

\end{abstract}

\section*{Introduction}
\markright{Introduction}
\noindent In this paper the class of all metric spaces considered up to an isometry and endowed with the Gromov--Hausdorff distance is investigated. Recall that this distance is defined as a measure of ``unlikeness'' of metric spaces. In fact, the distance is equal (up to multiplication by $2$) to the least possible distortion of metric over all correspondences (multivalued analogues of bijections) between those spaces. As any distance, it permits to define a convergence of sequences, and as this convergence of metric spaces, so as the distance itself are actively used in many different applications, such as, for example, groups growth velocities or images recognition and comparison.

If one restricts himself by compact metric spaces, then the Gromov--Hausdorff distance is a metric. The resulting metric space $\cM$ is referred as the Gromov--Hausdorff space, and many its properties are well-studied. For example, this space is path-connected, Polish (i.e., separable and complete), and geodesic. Also, the Gromov Criterion is well-known that gives a necessary and sufficient condition for a subset of the Gromov--Hausdorff space to be pre-compact.

Even a minimal extension of this space to the family of proper spaces leads to new effects. First, the distance between different spaces can be infinite. Second, the distance between non-isometric spaces can be zero. But the triangle inequality for the Gromov--Hausdorff distance remains valid for arbitrary metric spaces.

In the non-compact case, a modification of the Gromov--Hausdorff distance that is referred as the pointed convergence is traditionally considered. Namely, in each metric space one chooses a point, and the convergency is defined as the convergence of the balls of the same radius centered at these points. For example, by means of this convergence tangent cones and asymptotic cones with a vertex at a given point are defined. Also notice that in some recent papers, see for example~\cite{Herron}, there are constructed some distance functions that generate the pointed convergence. Moreover, these distance functions turn out to be metrics on the class of pointed proper metric spaces.

In the present paper we investigate the Gromov--Hausdorff distance in its initial definition on the class $\GH$ of all metric spaces considered up to an isometry. To work with such a ``monster--space'' we use the von Neumann--Bernays--G\"odel axioms system that permits to define correctly a distance function even on such a proper class. But to generate a topology on this class by means of this distance function in the standard way turns out to be impossible (see below). We suggest a way to avoid this obstacle by means of a filtration over cardinality. As a result, we define continuous curves in $\GH$ and show that the Gromov--Hausdorff distance is an intrinsic extended pseudometric, i.e\., the distance between any two points is equal to the infimum of the lengths of curves connecting these points.

The second part of the paper is devoted to geometry of metric segments in $\GH$, where a metric segment is defined as a class of points lying between two fixed ones. We show that a metric segment could be a proper class, not a set  (Conjecture: It is always the case). Besides, we study the possibility to extend a metric segment (to another one) beyond some its endpoint. This problem turns out to be rather non-trivial, and its complete solution is unknown even for the Gromov--Hausdorff space $\cM$. The key result is Theorem~\ref{thm:ext_extens} (see below) that gives some sufficient condition of non-extendability of a metric segment beyond one of its endpoints. The final Section~\ref{sec:examples} contains many examples. Let us emphasize an interesting Example~\ref{ex:k_points} based on Hadwiger Theorem solving Borsuk Problem in a particular case. At the end of this Section it is also shown that no metric segment, whose endpoints are bounded metric spaces, can be infinitely extended beyond both its endpoints.

The work is partly supported by RFBR (Project~19-01-00775-a) and by the MGU Program supporting scientific schools.

\section{Preliminaries}
\markright{\thesection.~Preliminaries}
\noindent Let $X$ be an arbitrary set. By $\#X$ we denote the cardinality of $X$, and let $\cP_0(X)$ stand for the set of all \textbf{non-empty\/} subsets of $X$. A \emph{distance function on a set $X$} is any symmetric mapping $d\:X\x X\to[0,\infty]$ vanishing at the pairs of coinciding elements. If $d$ satisfies the triangle inequality, then $d$ is referred as a \emph{extended pseudometric}. If in addition $d(x,y)>0$ for all $x\ne y$, then $d$ is called an \emph{extended metric}. If $d(x,y)<\infty$ for all $x,y\in X$, then such distance function is called a \emph{metric}, and sometimes a \emph{finite metric\/} to highlight the difference from an extended metric. A set $X$ with an (extended) (pseudo-)metric is called an \emph{\(extended\/\) \(pseudo-\/\)metric space}.

If $X$ is a set with some distance function, then, as a rule, we denote the distance between points $x$ and $y$ by $|xy|$. If $\g$ is a curve in $X$, then by $|\g|$ we denote its length. Further, let $x,y\in X$ and $|xy|<\infty$, then we say that a point $z\in X$ \emph{lies between $x$ and $y$} if $|xz|+|zy|=|xy|$.  The set of all $z$ lying between $x$ and $y$ is called the \emph{metric segment between $x$ and $y$} and is denoted by $[x,y]$. If $|xy|>0$, then the metric segment is called \emph{non-degenerate}.

\begin{dfn}
We say that a metric segment $[x,y]$ can be \emph{extended beyond the point $y$} if there exists a point $z\in X$ such that $[x,z]$ contains $y$ and $|yz|>0$.
\end{dfn}

\begin{rk}
If $|xy|=0$, then any point $z$ satisfies $|zy|=|zx|=|zx|+|xy|=|zy|+|yx|$, and if $|zx|>0$, then $z$ gives an extension of the metric segment $[x,y]$ as beyond $x$, so as beyond $y$. We are mostly interested in extendability of non-degenerate metric segments.
\end{rk}

Let $X$ be a metric space. For any $A,\,B\in\cP_0(X)$ and $x\in X$ we put
\begin{flalign*}
\indent&|xA|=|Ax|=\inf\bigl\{|xa|:a\in A\bigr\},\quad |AB|=\inf\bigl\{|ab|:a\in A,\,b\in B\bigr\},&\\
\indent&d_H(A,B)=\max\{\sup_{a\in A}|aB|,\sup_{b\in B}|Ab|\}=\max\bigl\{\sup_{a\in A}\inf_{b\in B}|ab|,\,\sup_{b\in B}\inf_{a\in A}|ba|\bigr\}.
\end{flalign*}
The function $d_H\:\cP_0(X)\x\cP_0(X)\to[0,\infty]$ is called the \emph{Hausdorff distance}. It is well-known~\cite{BurBurIva} that $d_H$ is a metric on the family $\cH(X)\ss\cP_0(X)$ of all non-empty closed bounded subsets of $X$.

Let $X$ and $Y$ be metric spaces. A triplet $(X',Y',Z)$ consisting of a metric space $Z$ and two its subsets $X'$ and $Y'$ isometric to $X$ and $Y$, respectively, is called a \emph{realization of the pair $(X,Y)$}. The \emph{Gromov--Hausdorff distance $d_{GH}(X,Y)$ between $X$ and $Y$} is defined as the infimum of values $r$ for which there exists a realization $(X',Y',Z)$ of the pair $(X,Y)$ such that $d_H(X',Y')\le r$.

Notice that the Gromov--Hausdorff distance could take as finite so as infinite values, and it always satisfies the triangle inequality, see~\cite{BurBurIva}. Besides, this distance equals zero for any pair of isometric spaces, therefore, due to the triangle inequality, the Gromov--Hausdorff distance is well-defined on isometry classes of metric spaces (it does not depend on a choice of representatives in the classes). There are some examples of non-isometric metric spaces with zero Gromov--Hausdorff distance~\cite{Ghanaat}.

Since any non-empty set can be endowed with some metric (for example, one can put all non-zero distances to be equal to $1$), so there are ``as many'' isometry classes of metric spaces as all possible sets, i.e., the family of isometry classes is not a set, but it is a proper class which, together with the  Gromov--Hausdorff distance, is denoted by $\GH$. Here we use the concept of \emph{class\/} in the sense of von Neumann--Bernays--G\"odel Set Theory (NBG). Recall some concepts of the NBG.

In NBG all the objects (analogues of usual sets) are referred as \emph{classes}. There are two types of classes: the \emph{sets\/} and the \emph{proper classes}. An example of a proper class is the class of all sets. According to the G\"odel construction, one can distinguish a set from a proper class as follows: for a set there always exists a class containing this set as an element. For a proper class there is no such a class. So, elements of each class are sets. For the classes many standard operations are defined, for example, intersections, complements, products, mappings, etc. Von Neumann Theorem says that a class is a proper one, iff it can be surjectively mapped onto the class of all sets.

Notice that such concepts as the \emph{distance function}, the \emph{\(extended\/\) pseudometric}, and the \emph{\(extended\/\) metric\/} are defined for any class, as for a class that is a set, so as for a class that is a proper one, because the products and the mappings are defined for all classes. However, definitions of other structures on proper classes could face some difficulties. For example, if one tries to define a topology on some proper class $\gC$ in the usual way, then $\gC$ itself must be an element of the topology, therefore, $\gC$ must be a set, a contradiction.

To avoid this problem, we consider a ``filtration'' of a class $\gC$ by its subclasses $\gC_n$ each of which consists of all the elements from $\gC$, whose cardinality is at most $n$, where $n$ is a cardinal. Recall that elements of any class are sets, and hence, the concept of cardinality is well-defined for them. The main examples for us are the class $\GH$ defined above, and the class $\cB$ of all bounded metric spaces considered up to an isometry. Notice that for any cardinal $n$ the subclasses $\GH_n$ and $\cB_n$ are sets. Indeed, the family of all cardinals that do not exceed a given one is a set, and for any fixed cardinal $n$ the family of all isometry classes of metric spaces of cardinality $n$ ``can not be greater'' than the set of all subsets of $X\x X\x\R$, where $X$ is an arbitrary set of cardinality $n$.

We say that a class $\gC$ is \emph{set-filtered\/} if all its subclasses $\gC_n$ are sets. Evidently, if a class $\gC$ is a set, then it is set-filtered.

Thus, let $\gC$ be a set-filtered class. We say that this class satisfies some properties, if these properties are valid for each set $\gC_n$. Let us give some examples.
\begin{itemize}
\item Let a distance function on $\gC$ be given, then it induces a ``usual'' distance function on each set $\gC_n$. Thus, in each  $\gC_n$ all concepts of metric geometry are defined (see above), for example open balls, and they are sets. The latter gives an opportunity to define the metric topology $\t_n$ on $\gC_n$ taking those balls as a base of topology. Clear that for $n\le m$ the topology $\t_n$ is induced by the topology $\t_m$.
\item Moreover, we say that a \emph{topology\/} on $\gC$ is defined, if for any cardinal $n$ a topology $\t_n$ is defined on $\gC_n$ in such a way that the following \emph{consistency condition\/} holds: if $n\le m$, then $\t_n$ is the topology induced by $\t_m$.
\item The presence of some topology on a class $\gC$ makes it possible to define, for example, continuous mappings from a topological space $Z$ to the class $\gC$. Notice that, in accordance with NBG axioms, for an arbitrary mapping $f\:Z\to\gC$ from the set $Z$ to the class $\gC$, the image $f(Z)$ is a set, all whose elements are also some sets, and their union $\cup f(Z)$ is a set of some cardinality $n$. Therefore each element from $f(Z)$ has a cardinality at most $n$, and hence, $f(Z)\ss\gC_n$. We call the mapping  $f$ \emph{continuous\/} if $f$ is continuous as a mapping from $Z$ to the topological space $\gC_n$. The consistency condition implies that for any $m\ge n$ the mapping $f$ is a continuous mapping from $Z$ to $\gC_m$ also, and for any $k\le n$ such that $f(Z)\ss\gC_k$ the mapping $f|_{\gC_k}$ is continuous too.
\item The above construction gives us an opportunity to define \emph{continuous curves in a class $\gC$} endowed with a topology.
\item Let a distance function and the corresponding topology be given on a class $\gC$. We say that this distance function is \emph{intrinsic\/} if it satisfies the triangle inequality, and for any two elements $x,\,y\in\gC$ such that $|xy|<\infty$ the distance $|xy|$ equals the infimum of lengths of curves connecting $x$ and $y$. Below we show that the Gromov--Hausdorff distance is intrinsic as on the class $\GH$, so as on the class $\cB$.
\end{itemize}

The most well-investigated subset of $\GH$ is the set of isometry classes of compact metric spaces. This set is called the \emph{Gromov--Hausdorff space\/} and is often denoted by $\cM$. It is well-known~\cite{BurBurIva,IvaNikolaevaTuz} that the restriction of the Gromov--Hausdorff distance to $\cM$ is a metric, and the metric space $\cM$ is Polish and geodesic.

Notice that to simplify notation, it is convenient not to distinguish isometry classes of metric spaces from their representatives. We have already used such convention, namely, we have defined $\cB$ as the class of bounded metric spaces \emph{considered up to an isometry}. Below we use this identification more than once, and write $X\in\GH$ having in mind that $X$ is some specific metric space.

As a rule, calculation of the Gromov--Hausdorff distance between specific metric spaces is rather hard problem, and for today it is known for a few pairs of spaces, see for example~\cite{GrigIT_Sympl}. The most useful tool for calculation of such kind is the following equivalent definition of the Gromov--Hausdorff distance. Recall that a \emph{relation\/} between sets $X$ and $Y$ is defined as a subset of their Cartesian product $X\x Y$. So, $\cP_0(X\x Y)$ is the set of all non-empty relations between $X$ and $Y$.

\begin{dfn}
For any $X,Y\in\GH$ and any $\s\in\cP_0(X\x Y)$ the value
$$
\dis\s=\sup\Bigl\{\bigl||xx'|-|yy'|\bigr|:(x,y),\,(x',y')\in\s\Bigr\}
$$
is called the \emph{distortion of the relation $\s$}.
\end{dfn}

A relation $R\ss X\x Y$ between sets $X$ and $Y$ is called a \emph{correspondence\/} if the canonical projections $\pi_X\:(x,y)\mapsto x$ and $\pi_Y\:(x,y)\mapsto y$ are surjective on $R$. By $\cR(X,Y)$ we denote the set of all correspondences between $X$ and $Y$.

\begin{thm}[\cite{BurBurIva}]\label{th:GH-metri-and-relations}
For any $X,Y\in\GH$ the equality
$$
d_{GH}(X,Y)=\frac12\inf\bigl\{\dis R:R\in\cR(X,Y)\bigr\}
$$
holds.
\end{thm}

In what follows we need the following estimates on the Gromov--Hausdorff distance that can be easily verified by means of Theorem~\ref{th:GH-metri-and-relations}. By $\D_1$ we denote the single-point metric space.

\begin{ass}\label{ass:estim}
For any $X,Y\in\GH$ the following relations are valid
\begin{itemize}
\item $2d_{GH}(\D_1,X)=\diam X$\rom;
\item $2d_{GH}(X,Y)\le\max\{\diam X,\diam Y\}$\rom;
\item If either $X$ or $Y$ has a finite diameter, then $\bigl|\diam X-\diam Y\bigr|\le2d_{GH}(X,Y)$.
\end{itemize}
\end{ass}

\begin{cor}
If $X,Y\in\cB$, then $[X,Y]$ is defined and $[X,Y]\ss\cB$.
\end{cor}

For topological spaces $X$ and $Y$, we consider $X\x Y$ as a topological space with the standard topology of Cartesian product. So, \emph{closed relations\/} and \emph{closed correspondences\/} are defined.

A correspondence $R\in\cR(X,Y)$ is called \emph{optimal\/} if $2d_{GH}(X,Y)=\dis R$. By $\cR_\opt(X,Y)$  we denote the set of all optimal correspondences between $X$ and $Y$, and by $\cR_\opt^c(X,Y)$ we denote the subset of $\cR_\opt(X,Y)$ consisting of all closed optimal correspondences.

\begin{thm}[\cite{Memoli,IvaIliadisTuz}]\label{thm:optimal-correspondence-exists}
For any $X,\,Y\in\cM$ there exists a closed optimal correspondence and also a realisation $(X',Y',Z)$ of the pair $(X,Y)$ which the Gromov--Hausdorff distance between $X$ and $Y$ attained at.
\end{thm}

\begin{thm}[\cite{Memoli,IvaIliadisTuz}]
For any $X,\,Y\in\cM$ and any $R\in\cR^c_\opt(X,Y)$ the family $R_t$, $t\in[0,1]$, of compact metric spaces, where $R_0=X$, $R_1=Y$, and for $t\in(0,1)$ the space $R_t$ is the set $R$ with the metric
$$
\bigl|(x,y),(x',y')\bigr|_t=(1-t)|xx'|+t\,|yy'|,
$$
is a shortest curve in $\cM$ connecting $X$ and $Y$, and the length of this curve equals $d_{GH}(X,Y)$.
\end{thm}

We will also use the following notations. Let $X$ be an arbitrary set and $m$ a cardinal such that $1<m\le\#X$. By $\cC_m(X)$ we denote the family of all possible coverings of the set $X$ by its $m$ non-empty subsets, and by $\cD_m(X)$ we denote the family of all partitions of $X$ into $m$ non-intersecting subsets. Obviously, $\cD_m(X)\ss\cC_m(X)$. If $X$ is a metric space, then for any $D=\{X_i\}_{i\in I}\in\cC_m(X)$ we put
$$
\diam D=\sup_{i\in I}\diam X_i,\quad \a(D)=\inf\bigl\{|X_iX_j|:i\ne j\bigr\}
$$
and define the values
$$
d_m(X)=\inf_{D\in\cD_m(X)}\diam D,\quad \a_m(X)=\sup_{D\in\cD_m(X)}\a(D).
$$

\section{The Gromov--Hausdorff Distance is Intrinsic}
\markright{\thesection.~The Gromov--Hausdorff Distance is Intrinsic}
\noindent Let $\gC$ be a set-filtered class endowed with a distance function that satisfies the triangle inequality. Let $x,y\in\gC$, $|xy|<\infty$, and $\g$ be a curve in $\gC$ connecting $x$ and $y$. A curve $\g$ is said to be an \emph{$\e$-shortest\/} for $x$ and $y$ if $0\le|\g|-|xy|\le\e$. It is easy to see that a distance function on $\gC$ that satisfies the triangle inequality is intrinsic if and only if for any pair of elements $x,\,y$ from $\gC$ such that $|xy|<\infty$, and for any $\e>0$, there exists an $\e$-shortest curve connecting $x$ and $y$. We use this reasonings to prove the following Theorem.

\begin{thm}\label{thm:inter_metric}
Let $X$ and $Y$ be arbitrary metric spaces such that $d_{GH}(X,Y)<\infty$. Let $R\in\cR(X,Y)$ be an arbitrary correspondence such that
$$
\dis R-2d_{GH}(X,Y)\le2\e.
$$
Then the family $R_t$, $t\in[0,1]$, of metric spaces, where $R_0=X$, $R_1=Y$, and for $t\in(0,1)$ the space $R_t$ is the set $R$ endowed with the metric
$$
\bigl|(x,y),(x',y')\bigr|_t=(1-t)|xx'|+t\,|yy'|,
$$
is an $\e$-shortest curve in $\GH$ connecting $X$ and $Y$. Moreover, if $X$ and $Y$ are bounded spaces, then all the spaces $R_t$ are also bounded, i.e., the curve $R_t$ is an $\e$-shortest curve in $\cB$.
\end{thm}

\begin{proof}
Put $n=\#R$ and consider the following correspondences $R_X\ss X\times R$ and $R_Y\ss R\times Y$ between $X$ and $R_t$ and between $R_t$ and $Y$, respectively:
$$
R_X=\bigl\{\bigl(x,(x,y)\bigr):x\in X,\,(x,y)\in R\bigr\},\
R_Y=\bigl\{\bigl((x,y),y\bigr):y\in Y,\,(x,y)\in R\bigr\}.
$$
Then $\dis R_X=t\dis R$ and $\dis R_Y=(1-t)\dis R$. Further, taking the identical correspondence between the spaces $R_t$ and $R_s$, where $s,\,t\in(0,1)$, we get $2d_{GH}(R_t,R_s)\le|t-s|\dis R$. Put $\g(t)=R_t$. The above implies that $\g$ is a continuous mapping from $[0,1]$ to $\GH_n$, i.e\., $\g$ is a continuous curve in the class $\GH$. Besides, $2|\g|\le\dis R$, therefore $|\g|\le d_{GH}(X,Y)+\e$, and hence, $\g$ is an $\e$-shortest curve for $X$ and $Y$. It remains to notice that $\diam R_t\le\max\{\diam X,\diam Y\}$, and so, if $X,\,Y\in\cB$, then $R_t\in\cB$ for all $t$ as well.
\end{proof}

\begin{cor}\label{cor:intrinsic}
The Gromov--Hausdorff distance on the class $\GH$ is an intrinsic extended pseudometric. It is an intrinsic finite pseudometric on the class $\cB$.
\end{cor}

We call the $\e$-shortest curve constructed in Theorem~\ref{thm:inter_metric} \emph{linear}.

\begin{rk}
The proof of Theorem~\ref{thm:inter_metric} implies that the linear $\e$-shortest curve connecting metric spaces $X$ and $Y$, $d_{GH}(X,Y)<\infty$, is a Lipschitz curve in $\GH$, and $d_{GH}(X,Y)+\e$ is its Lipschitz constant.
\end{rk}

\section{Metric Segments and their Extendability}
\markright{\thesection.~Metric Segments and their Extendability}
\noindent We start with a description of simple properties of $\e$-shortest curves in an extended pseudometric space $\gC$, where $\gC$ is a set-filtered class.

\begin{lem}\label{lem:eps-geod}
Let $x,\,y\in\gC$, $|xy|<\infty$, $\g$ be an $\e$-shortest curve connecting $x$ and $y$, and $w$ be an arbitrary point of the curve $\g$. Then the segments $\g_{xw}$ and $\g_{wy}$ of the curve $\g$ between the points $x$, $w$ and  $w$, $y$, respectively, are $\e$-shortest curves for their endpoints, and, moreover, the inequality $|xw|+|wy|-|xy|\le\e$ holds.
\end{lem}

\begin{proof}
Indeed, by the definition of an $\e$-shortest curve, additivity of the length of a curve, and the triangle inequality we have:
$$
\e\ge|\g|-|xy|=|\g_{xw}|+|\g_{wy}|-|xy|\ge\big(|\g_{xw}|-|xw|\big)+\big(|\g_{wy}|-|wy|\big).
$$
Notice that the expressions in parentheses are non-negative, therefore each of them does not exceed $\e$, and hence,  the curves $\g_{xw}$ and $\g_{wy}$ are $\e$-shortest. Further,
$$
|xw|+|wy|\le|\g_{xw}|+|\g_{wy}|=|\g_{xy}|\le|xy|+\e.
$$
Lemma is proved.
\end{proof}

Generally speaking, the union of two $\e$-shortest curves is not an $\e$-shortest one. However, the following result holds.

\begin{lem}\label{lem:union_eps_geo}
Let $x,\,y\in\gC$, $|xy|<\infty$, and a point $w\in\gC$ lie between the points $x$ and $y$. Then the union of an $\e$-shortest curve $\g_{xw}$ for $x,\,w$ and a $\dl$-shortest curve $\g_{wy}$ for $w,\,y$, respectively, is an $(\e+\dl)$-shortest curve for $x,\,y$.
\end{lem}

\begin{proof}
Indeed, denote by $\g_{xy}$ the union of the curves $\g_{xw}$ and $\g_{wy}$. Then
$$
|xy|=|xw|+|wy|\ge|\g_{xw}|-\e+|\g_{wy}|-\dl=|\g_{xy}|-(\e+\dl).
$$
Lemma is proved.
\end{proof}

Lemmas~\ref{lem:eps-geod} and~\ref{lem:union_eps_geo} imply the following estimate.

\begin{cor}\label{cor:between}
Let $x,\,y\in\gC$, $|xy|<\infty$, and a point $w\in\gC$ lie between the points $x$ and $y$. Let $\g_{xw}$ and $\g_{wy}$ be $\e$-shortest curves connecting $x,\,w$ and $w,\,y$, respectively. Then for any points $p\in\g_{xw}$ and $q\in\g_{wy}$ the inequality $|py|+|yq|-|pq|\le2\e$ holds.
\end{cor}

Let us list several elementary properties of metric segments.

\begin{lem}\label{lem:metr_segm}
For any $x,y,z\in\gC$ we have
\begin{itemize}
\item If $y\in[x,z]$, then $[x,y]\ss[x,z]$ and $[y,z]\ss[x,z]$\rom;
\item If $y$ lies between $x$ and $z$, then $y$ lies also between any points $x'\in[x,y]$ and $z'\in[y,z]$.
\end{itemize}
\end{lem}

\begin{rk}
If the space $\gC$ is geodesic, then extendability of a metric segment $[x,y]$ beyond $y$ is equivalent to existence of a point $z\in\gC$ such that any shortest curve connecting $x$ and $y$ is contained in a shortest curve connecting $x$ and $z$.
\end{rk}

\begin{rk}
If $\gC$ is a proper class, then a metric segment could also be a proper class. As an example, consider the space $\GH$ and the metric segment $[\D_1,\D_n]$ in it, where $\D_n$ stands for the metric space of cardinality $n$ all whose non-zero distances are equal to $1$. It is easy to verify that the curve $\g(t)$, $t\in[0,1]$, where $\g(0)=\D_1$ and $\g(t)=t\,\D_n$ for all other $t$, is a shortest curve connecting $\D_1$ and $\D_n$. Consider the space $Z=\frac12\D_n$, choose its arbitrary point $z\in Z$, and change it with a non-empty set $A$ of an arbitrary cardinality. Fix a positive $\e<1/2$. On the resulting set $Z'=\bigl(Z\sm\{z\}\bigr)\sqcup A$ redefine the distance as follows: $|aa'|=\e$ and $|az'|=|zz'|=1/2$ for any distinct $a,\,a'\in A$ and any $z'\in Z\sm\{z\}$. It is easy to verify that $d_{GH}(Z',\D_1)=d_{GH}(Z,\D_1)$ and $d_{GH}(Z',\D_n)=d_{GH}(Z,\D_n)$, therefore, any such space $Z'$ lies between $\D_1$ and $\D_n$, i.e., it belongs to the metric segment $[\D_1,\D_n]$. Since the cardinality of the space $Z'$ is arbitrary, then $[\D_1,\D_n]$ is a proper class.
\end{rk}

Let $X\in\cB$, and $C=\{X_i\}_{i\in I}$ be a covering of $X$ by non-empty sets. We say that $\{X_i\}_{i\in I}$ is a covering by \emph{sets of less diameter than $X$} if $\diam C<\diam X$.

\begin{lem}\label{lem:near_covered}
Let $X\in\cB$, and $C_X=\{X_i\}_{i\in I}$ be a covering of $X$ by non-empty subsets of less diameter than $X$. Put $\dl_X=\diam X-\diam C_X>0$. Then any $Y\in\cB$ such that $d_{GH}(X,Y)<\e=\dl_X/5$, has a similar representation, i.e., there exists a covering $C_Y=\{Y_i\}_{i\in I}$ of the space $Y$ by non-empty subsets of less diameter, and $\diam Y-\diam C_Y>\e$.
\end{lem}

\begin{proof}
Since $d_{GH}(X,Y)<\e$, then there exists an $R\in\cR(X,Y)$ such that $\dis R<2\e$. For each $i\in I$ put $Y_i=R(X_i)$. Since $R$ is a correspondence, then $Y_i\ne\0$ for any $i\in I$, and $C_Y:=\{Y_i\}_{i\in I}$ is a covering of $Y$ by non-empty sets. Since $\dis R<2\e$, $\diam Y_i\le\diam X_i+\dis R$, and $\diam X\le\diam Y+\dis R$, then
\begin{multline*}
\diam C_Y\le\diam C_X+\dis R<\diam C_X+2\e=\diam X-\dl_X+2\e=\\
=\diam X-3\e\le\diam Y+\dis R-3\e<\diam Y-\e.
\end{multline*}
Lemma is proved.
\end{proof}

\begin{lem}\label{lem:near_covered_non_bound}
Let $X\in\GH$, and $C_X=\{X_i\}_{i\in I}$ be a covering of $X$ by non-empty subsets of finite diameter. Then any $Y\in\GH$ such that $d_{GH}(X,Y)<\infty$ has a similar representation, i.e., there exists a covering $C_Y=\{Y_i\}_{i\in I}$ of the space $Y$ by non-empty subsets of finite diameter, and  $\diam C_Y\le\diam C_X+2d_{GH}(X,Y)$.
\end{lem}

\begin{proof}
Let $d_{GH}(X,Y)<\e$, then there exists an $R\in\cR(X,Y)$ such that $\dis R<2\e$. For each $i\in I$ put $Y_i=R(X_i)$. Since $R$ is a correspondence, then $Y_i\ne\0$ for any $i\in I$, and $C_Y:=\{Y_i\}_{i\in I}$ is a covering of $Y$ by non-empty sets. Since $\dis R<2\e$, then $\diam Y_i\le\diam X_i+\dis R$ and $\diam C_Y\le\diam C_X+2\e$, that implies the statement of Lemma.
\end{proof}

By means of the following standard construction, one can pass form a covering to a partition whose cardinality and diameter do not increase.

\begin{constr}\label{constr:part}
Let $X\in\GH$, and $C=\{X_i\}_{i\in I}$ be a covering of $X$ such that $\diam C<\diam X$. By well-ordering Zermelo's theorem, we can introduce a strict total order on the index set $I$, and put
$$
X'_i=X_i\setminus \bigcup_{j:j<i}X_j,\quad i\in I.
$$
After eliminating all empty $X'_i$, we obtain a partition $D$ of the space $X$ into $m\le\#I$ non-empty subsets such that $\diam D\le\diam C<\diam X$.
\end{constr}

\begin{lem}\label{lem:dist_part}
Let $X,Y\in\cB$ be such that
\begin{enumerate}
\item $d:=\diam Y-\diam X>0$\rom;
\item There exists a partition $D_X=\{X_i\}_{i\in I}$ of the space $X$ with $\a(D_X)>0$\rom;
\item There exists a covering $C_Y=\{Y_j\}_{j\in J}$ of the space $Y$ by subsets of less diameter such that  $\dl_Y:=\diam Y-\diam C_Y>0$\rom;
\item $\#J\le\#I$.
\end{enumerate}
Then $\diam Y-2d_{GH}(X,Y)\ge\min\{d,\a(D_X),\dl_Y\}>0$.
\end{lem}

\begin{proof}
Using Construction~\ref{constr:part} described above, reconstruct the covering $C_Y=\{Y_j\}$ of the space $Y$ into a partition $D_Y=\{Y'_k\}_{k\in K}$, $\#K\le\#J\le\#I$, $\diam D_Y\le\diam C_Y$. Let $\s\:K\to I$ be an arbitrary injection. Consider the following correspondence:
$$
R=\Bigl(\bigcup_{k\in K\sm\{k_0\}}(X_{\s(k)}\times Y'_k)\Bigr)\bigcup \Bigl(\bigcup_{i\in I\sm\s(K)}(X_i\times Y'_{k_0})\Bigr),
$$
where $k_0\in K$ is any fixed element. Estimate the distortion of the correspondence $R$. Since elements form distinct $Y'_k$ always correspond in $R$ to elements from distinct $X_i$, then
\begin{multline*}
2d_{GH}(X,Y)\le\dis R\le\max\bigl\{\diam X,\,\diam Y-\a(D_X),\,\diam D_Y\bigr\}\le \\
\le\max\bigl\{\diam Y-d,\,\diam Y-\a(D_X),\,\diam Y-\dl_Y\bigr\}\le\\
\le\diam Y-\min\{d,\a(D),\dl_Y\}.
\end{multline*}
Lemma is proved.
\end{proof}

\begin{dfn}
Let $X,Y\in\cB$ be such that $d_{GH}(X,Y)\ne0$. We say that $Y$ is \emph{hyperextreme with respect to $X$} if
$$
2d_{GH}(X,Y)=\diam Y\ge\diam X,
$$
and that $Y$ is \emph{subextreme with respect to $X$} if
$$
2d_{GH}(X,Y)=\diam Y-\diam X.
$$
\end{dfn}

\begin{rk}
If $Y$ is hyperextreme with respect to $X$, then, due to Assertion~\ref{ass:estim}, the distance $d_{GH}(X,Y)$ between $X$ and $Y$ takes the greatest possible value for the spaces with such diameters. Conversely, if $2d_{GH}(X,Y)=\max\{\diam X,\diam Y\}>0$ and $\diam Y\ge\diam X$, then $Y$ is hyperextreme with respect to $X$.
\end{rk}

\begin{rk}
If $Y$ is subextreme with respect to $X$, then $\diam Y>\diam X$ and, due to Assertion~\ref{ass:estim}, the distance $d_{GH}(X,Y)$ takes the least possible value for the spaces $X$ and $Y$ with such diameters. Conversely, if $2d_{GH}(X,Y)=|\diam X -\diam Y|>0$, then the space with a larger diameter is subextreme with respect to the one with a smaller diameter.
\end{rk}

\begin{rk}
Any metric space $Y\in\cB$, $Y\ne\D_1$, is simultaneously hyperextreme and subextreme with respect to the single-point space $\D_1$. This is the only such case: if $Y$ is simultaneously subextreme and hyperextreme with respect to some $X$, then $2d_{GH}(X,Y)=\diam Y=\diam Y-\diam X$, therefore $\diam X=0$, and hence, $X=\D_1$.
\end{rk}

\begin{dfn}
If
$$
2d_{GH}(X,Y)=\diam Y=\diam X>0,
$$
then $Y$ is hyperextreme with respect to $X$, and $X$ is hyperextreme with respect to $Y$. In this case, we say that the spaces $X$ and $Y$ are \emph{mutually hyperextreme}, and the metric segment $[X,Y]\ss\cB$ is said to be \emph{extreme}.
\end{dfn}

\begin{rk}\label{rk:diam_radius}
Assertion~\ref{ass:estim} implies that the spaces $X$ and $Y$ are mutually hyperextreme if and only if they are ``diametrally opposite'' points of the sphere of radius $\frac12\diam X$ centered at the single-point space $\D_1$, i.e., $X$ and $Y$ are points of this sphere most distant from each other. To the contrary, if $Y$ is subextreme with respect to $X$, then $Y$ is the closest to $X$ point of the sphere of radius $\frac12\diam Y$ centered at $\D_1$, i.e., $X$ and $Y$ ``lie at the same radial ray''.
\end{rk}

\begin{ass}\label{ass:ext_extens}
Let $X,Y\in\cB$, and the metric segment $[X,Y]$ be extreme. Assume that $[X,Y]$ can be extended beyond $Y$ to some $Z$. Then $Z\in\cB$, $\diam Z>\diam Y$, and the space $Z$ is subextreme with respect to the space $Y$.
\end{ass}

\begin{proof}
Indeed, since $Y$ lies between $X$ and $Z$ due to assumptions, then, by definition, $d_{GH}(X,Z)<\infty$, and hence,  $Z\in\cB$ in accordance with Assertion~\ref{ass:estim}. Moreover, $d_{GH}(X,Z)=d_{GH}(X,Y)+d_{GH}(Y,Z)$, and $d_{GH}(Y,Z)>0$ by definition of the extendability, so $d_{GH}(X,Z)> d_{GH}(X,Y)$. On the other hand, if $\diam Z\le\diam Y$, then, in accordance with Assertion~\ref{ass:estim}, we have
$$
2d_{GH}(X,Z)\le\max\{\diam X,\diam Z\}=\diam X,
$$
and $\diam X=2d_{GH}(X,Y)$ because the segment $[X,Y]$ is extreme, and so $d_{GH}(X,Y)=d_{GH}(X,Z)$, a contradiction. Thus, $\diam Z>\diam Y$.

Further, using the assumptions that $Y$ lies between $X$ and $Z$, that the segment $[X,Y]$ is extreme, and Assertion~\ref{ass:estim}, we have:
\begin{multline*}
\max\big\{\diam X,\diam Z\big\}\ge2d_{GH}(X,Z)=2d_{GH}(X,Y)+2d_{GH}(Y,Z)=\\
=\diam X+2d_{GH}(Y,Z)\ge\diam X+\big|\diam Y-\diam Z\big|.
\end{multline*}
However, since $\diam Z\ge\diam Y=\diam X$, then the leftmost part and the rightmost part of this relations chain equal to $\diam Z$, so $2d_{GH}(Y,Z)=\diam Z-\diam Y>0$.
\end{proof}

\begin{lem}\label{lem:hypo_lin_geod}
Let $Y,Z\in\cB$, and $Z$ be subextreme with respect to $Y$. Fix an arbitrary $\e$ from the interval $(0,\diam Z-\diam Y)$ and consider a linear $\e$-shortest curve $R_t$. Then
$$
\diam R_t\ge(1-t)\diam Y+t\,\diam Z-4\e.
$$
\end{lem}

\begin{proof}
Let $R\in\cR(Y,Z)$ be an arbitrary correspondence such that $\dis R\le2d_{GH}(Y,Z)+2\e$. Choose  $z,\,z'\in Z$ such that $|zz'|\ge\diam Z-\e$, and hence, $|zz'|>\diam Y$. Then for any $y\in R^{-1}(z)$ and $y'\in R^{-1}(z')$ we have:
$$
\big||zz'|-|yy'|\big|\le\dis R\le2d_{GH}(Y,Z)+2\e=\diam Z-\diam Y+2\e,
$$
where the latter equality holds because $Z$ is subextreme with respect to $Y$. On the other hand,  $\big||zz'|-|yy'|\big|=|zz'|-|yy'|\ge\diam Z-\e-|yy'|$ due to the choice of $z,\,z'$ and $\e$, therefore,
$$
\diam Z-\diam Y+2\e\ge\diam Z-\e-|yy'|,
$$
and hence, $\diam Y-|yy'|\le3\e$. Thus,
\begin{multline*}
\diam R_t\ge\big|(y,z),(y',z')\big|_t=(1-t)\,|yy'|+t\,|zz'|\ge\\
\ge(1-t)(\diam Y-3\e)+t\,(\diam Z-\e)\ge\\
\ge(1-t)\diam Y+t\,\diam Z-4\e.
\end{multline*}
Lemma is proved.
\end{proof}

\begin{thm}\label{thm:ext_extens}
Let $X,Y\in\cB$, $m$ and $n$ be cardinal numbers such that $1<n\le\#X$, $1<m\le\#Y$. Suppose that the metric segment $[X,Y]$ is extreme and the following conditions hold\/\rom:
\begin{enumerate}
\item There exists a partition $D_X\in\cD_n(X)$ such that $\a(D_X)>0$\rom;
\item There exists a covering $C_Y\in\cC_m(Y)$, such that $\diam C_Y<\diam Y$\rom;
\item $m\le n$.
\end{enumerate}
Then the metric segment $[X,Y]$ can not be extended beyond $Y$.
\end{thm}

\begin{proof}
To the contrary assume that there exist a metric space $Z$ such that $Y$ lies between $X$ and $Z$. Due to Assertion~\ref{ass:ext_extens}, in this case $Z\in\cB$ and $Z$ is subextreme with respect to $Y$. The following Lemma is evident.

\begin{lem}\label{lem:quadr_ineq}
For any $\dl>0$ and $d>0$ there exists an $\e_0>0$ such that the inequality
\begin{equation}\label{eq:eddl}
\frac{8\e}{d}<\frac{\dl}{10(d+\e)}
\end{equation}
holds for all $\e\in[0,\e_0)$.
\end{lem}

Under notations of Lemma~\ref{lem:quadr_ineq}, choose the corresponding $\e_0$ for $\dl=\diam Y-\diam C_Y$ and $d=d_{GH}(Y,Z)=\diam Z-\diam Y>0$. Fix an arbitrary $\e$ such that
$$
0<\e<\min\biggl\{\e_0,\,\frac d8,\,\frac\dl{20},\,\frac{\a(D_X)}4\biggr\},
$$
then, due to Lemma~\ref{lem:quadr_ineq}, Inequality~(\ref{eq:eddl}) holds. Since $\e<d/8$, then the left hand side of Inequality~(\ref{eq:eddl}) is less than $1$, therefore there exists an $s\in(0,1)$ such that
$$
\frac{8\e}{d}<s<\frac{\dl}{10(d+\e)}.
$$

Consider any $R\in\cR(Y,Z)$ with $\dis R\le2d+2\e$ and the corresponding linear $\e$-shortest curve $R_t$, $t\in[0,1]$, connecting $Y$ and $Z$, where $R_0=Y$, $R_1=Z$. For the $s$ chosen above, consider the space $R_s$ and by reasoning similar to the proof of Theorem~\ref{thm:inter_metric}, we get that $2d_{GH}(Y,R_s)\le s\dis R$, and so,
$$
d_{GH}(Y,R_s)\le s(2d+2\e)<\frac{\dl}{10(d+\e)}(2d+2\e)=\dl/5.
$$
Due to Lemma~\ref{lem:near_covered}, there exists a covering $C_R$ of the space $R_s$ by non-empty subsets of less diameter, such that $\#C_R=\#C_Y$ and $\diam R_s-\diam C_R>\dl/5$. Since $\diam Y<\diam Z$ and $8\e<s\,d$, than, by Lemma~\ref{lem:hypo_lin_geod}, we get
$$
\diam R_s\ge\diam Y+s(\diam Z-\diam Y)-4\e=\diam Y+s\,d-4\e>\diam Y+4\e,
$$
and so,  because $\diam X=\diam Y$, we have $\diam R_s-\diam X>4\e>0$. The latter, together with the above estimates,  imply that the pair $X$ and $R_s$ satisfies all conditions of Lemma~\ref{lem:dist_part}, and so,
\begin{multline*}
2d_{GH}(X,R_s)\le\\
\le\diam R_s-\min\big\{\diam R_s-\diam X,\,\a(D),\,\diam R_s-\diam C_R\big\}.
\end{multline*}
However, each of the three expressions standing under the minimum in the previous formula are strictly greater than $4\e$, and hence, $2d_{GH}(X,R_s)<\diam R_s-4\e$.

On the other hand, $2d_{GH}(Y,R_s)\ge\big|\diam Y-\diam R_s|=\diam R_s-\diam Y$, and so, applying Corollary~\ref{cor:between} and taking into account the condition $2d_{GH}(X,Y)=\diam Y$, we obtain that
\begin{multline*}
2d_{GH}(X,R_s)\ge2d_{GH}(X,Y)+2d_{GH}(Y,R_s)-4\e\ge\\
\ge2d_{GH}(X,Y)+\diam R_s-\diam Y-4\e=\diam R_s-4\e,
\end{multline*}
a contradiction. Theorem is proved.
\end{proof}

\begin{rk}
One might get the impression that the statement of Theorem~\ref{thm:ext_extens} can be weaken by changing the mutual hyperextremality of the space $X$ and $Y$ by hyperextremality of $Y$ with respect to $X$. But if $Y$ is hyperextreme with respect to $X$ and Conditions~$(1)$--$(3)$ of Theorem~\ref{thm:ext_extens} are valid, then the spaces $X$ and $Y$ are mutually extreme. Indeed, if  $\diam Y>\diam X$, then Lemma~\ref{lem:dist_part} implies the inequality $2d_{GH}(X,Y)<\diam Y$ that contradicts to hyperextremality of $Y$ with respect to $X$, and hence only the case $2d_{GH}(X,Y)=\diam Y=\diam X$ is possible, but it is the case of mutual extremality.
\end{rk}

\begin{cor}\label{cor:ext_ext}
Let $X,\,Y\in\cM$, the spaces $X$ and $Y$ be mutual hyperextreme, and Conditions $(1)-(3)$ of Theorem~$\ref{thm:ext_extens}$ hold. Then any shortest curve in $\cM$ connecting $X$ and $Y$ can not be extended beyond $Y$.
\end{cor}

\section{Some Examples}\label{sec:examples}
\markright{\thesection.~Some Examples}
\noindent Above, by $\D_1$ we denoted the single-point metric space, and $\l\D_n$ was the metric space of cardinality $n$, all whose non-zero distances are equal to $\l$. Such spaces are called \emph{one distance spaces\/} or \emph{simplexes\/} for brevity. Some formulas for calculating the Gromov--Hausdorff distances from the simplexes to bounded metric spaces can be found in~\cite{GrigIT_Sympl}. We use these results to construct examples of metric segments that can be extended.

\subsection{Extendability beyond a Simplex}
\noindent Let $X$ be a metric space and $m$ a cardinal number that does not exceed $\#X$. Recall that we have already defined several characteristics of possible partitions of a metric space $X$ into $m$ subsets.

\begin{thm}[\cite{GrigIT_Sympl}]\label{thm:dist-n-simplex-same-dim-alpha0}
For $X\in\cB$ and a cardinal number $1<m\le\#X$ the equality
$$
2d_{GH}(\l\D_m,X)=\inf_{D\in\cD_m}\max\bigl\{\diam D,\,\l-\a(D),\,\diam X-\l\bigr\}
$$
is valid.
\end{thm}

Here we need  a particular case of this formula.

\begin{cor}\label{cor:formula_simplex}
For $X\in\cB$, a cardinal number $1<m\le\#X$, and $\l\ge\diam X+\a_m(X)$ the equality $2d_{GH}(\l\D_m,X)=\l-\a_m(X)$ holds.
\end{cor}

\begin{proof}
Indeed, if $\l\ge\diam X+\a_m(X)$, then $\l-\a(D)\ge\l-\a_m(X)\ge\diam X\ge\diam D$ and hence, $\l-\a(D)\ge\diam X-\l$. Therefore, for such $\l$ the maximum in the formula from Theorem~\ref{thm:dist-n-simplex-same-dim-alpha0} equals $\l-\a(D)$, so $2d_{GH}(\l\D_m,X)=\inf_D(\l-\a(D))=\l-\sup_D\a(D)=\l-\a_m(X)$. Corollary is proved.
\end{proof}

\begin{cor}\label{cor:extendable}
For $X\in\cB$, a cardinal number $1<m\le\#X$, and $\l\ge\diam X+\a_m(X)$ the metric segment $[X,\l\D_m]$ can be extended beyond $\l\D_m$ to any simplex $\l'\D_m$, where $\l'>\l$.
\end{cor}

\begin{proof}
Indeed, since $2d_{GH}(\l\D_m,\l'\D_m)=|\l-\l'|$ and in accordance with Corollary~\ref{cor:formula_simplex}, for any $\l\ge\diam X+\a_m(X)$ the equality $2d_{GH}(\l\D_m,X)=\l-\a_m(X)$ holds.  Therefore, the simplex $\l\D_m$ lies between $X$ and $\l'\D_m$, where $\l'>\l$, as required.
\end{proof}

\subsection{Extendability beyond a Subextreme Space}
\noindent The next Assertion formalizes Remark~\ref{rk:diam_radius}.

\begin{ass}\label{ass:hypo_between}
Let $X,Y\in\cB$, and $Y$ be subextreme with respect to $X$. Then $X$ lies between $\D_1$ and $Y$.
\end{ass}

\begin{proof}
Indeed,
\begin{multline*}
2d_{GH}(\D_1,Y)=\diam Y=\diam X +\big(\diam Y-\diam X\big)=\\
=2d_{GH}(\D_1,X)+2d_{GH}(X,Y),
\end{multline*}
that is required.
\end{proof}

\begin{cor}\label{cor:hypo_exten}
Let $X,Y\in\cB$, and $Y$ be subextreme with respect to $X$, and $X\ne\D_1$. Then the metric segment $[X,Y]$ is extendable as beyond $X$, so as beyond $Y$.
\end{cor}

\begin{constr}\label{constr:deux_point_extt}
Let $Y\in\cB$ and there exist points $y_1,y_2\in Y$ such that $\diam Y=|y_1y_2|$. For arbitrary $r_1\ge0$ and $r_2\ge0$, construct a \emph{two-point extension $Z_{r_1, r_2}(y_1, y_2)$ of the space $Y$} as follows. Put $Z=Z_{r_1,r_2}(y_1,y_2)=Y\sqcup\{z_1,z_2\}$ and extend the metric from $Y$ to $Z$ by the next rule: $|z_1z_2|=r_1+|y_1y_2|+r_2$, and for $y\in Y$ put $|yz_i|=r_i+|yy_i|$, $i=1,2$. If $r_i=0$, then identify $z_i$ with $y_i$. It is easy to verify that $Z_{r_1,r_2}(y_1,y_2)$ is a metric space for any $r_1\ge0$ and $r_2\ge0$, and
$$
\diam Z_{r_1, r_2}(y_1,y_2)=\diam Y+r_1+r_2=|z_1z_2|.
$$
The spaces
$$
Z_{r_1}(y_1):=Z_{r_1,0}(y_1,y_2)\qquad \text{and}\qquad Z_{r_2}(y_2):=Z_{0,r_2}(y_1, y_2)
$$
we call a \emph{single-point extensions of the space $Y$}. Clear that
$$
Z_{0,0}(y_1, y_2)=Z_{0}(y_1)=Z_{0}(y_2)=Y.
$$
\end{constr}

\begin{lem}\label{lem:z_rr_curve}
Let $Y\in\cB$ and there exist points $y_1,y_2\in Y$ such that $\diam Y=|y_1y_2|$. Let  $Z_{r_1,r_2}(y_1,y_2)$ be a two-point extension of the space $Y$ constructed above. Then for the curve  $\g(t)=Z_{r_1t,r_2t}(y_1,y_2)$, $t\in[0,1]$, the equality
$$
2d_{GH}\big(\g(t_1),\g(t_2)\big)=|t_1-t_2|(r_1+r_2)
$$
holds for all $t_1,t_2\in[0,1]$, and hence, $\g$ is a shortest curve connecting $Y$ and $Z_{r_1,r_2}(y_1,y_2)$. In particular, putting $r_2=0$ or $r_1=0$, one gets a shortest curve connecting $Y$ with the corresponding single-point extension $Z_{r_1}(y_1)$ or $Z_{r_2}(y_2)$.
\end{lem}

\begin{proof}
Consider the correspondence $R\in\cR\bigl(\g(t_1),\g(t_2)\bigr)$ that is identical on the set $Y\cup\{z_1,z_2\}$. Then
$$
\dis R=\max\bigl\{|t_1-t_2|r_1,\,|t_1-t_2|r_2,\,|t_1-t_2|(r_1+r_2)\bigr\}=|t_1-t_2|(r_1+r_2),
$$
and so, $2d_{GH}\big(\g(t_1),\g(t_2)\big)\le|t_1-t_2|(r_1+r_2)$. On the other hand,
\begin{multline*}
\bigl|\diam\g(t_1)-\diam\g(t_2)\bigr|=\\
=\bigl|\diam Y+(r_1+r_2)t_1-\diam Y-(r_1+r_2)t_2\bigr|=|t_1-t_2|(r_1+r_2),
\end{multline*}
and hence, due to Assertion~\ref{ass:estim}, we have $2d_{GH}\big(\g(t_1),\g(t_2)\big)\ge|t_1-t_2|(r_1+r_2)$, so
$$
2d_{GH}\big(\g(t_1),\g(t_2)\big)=|t_1-t_2|(r_1+r_2),
$$
that implies Lemma's statement.
\end{proof}

\begin{ass}\label{ass:hypo_exten_deux}
Let $X,Y\in\cB$, and $Y$ be subextreme with respect to $X$. Assume that there exist $y_1,y_2\in Y$ such that $\diam Y=|y_1y_2|$. Then the metric segment $[X,Y]$ is extendable beyond $Y$ to any two-point extension $Z:=Z_{r_1,r_2}(y_1, y_2)$ of the space $Y$.
\end{ass}

\begin{proof}
By Lemma~\ref{lem:z_rr_curve}, we have $2d_{GH}(Y,Z)=r_1+r_2$.

Now, estimate $d_{GH}(X,Z)$ by means of Assertion~\ref{ass:estim} and the triangle inequality:
\begin{multline*}
\diam Z-\diam X\le 2d_{GH}(X,Z)\le 2d_{GH}(X,Y)+2d_{GH}(Y,Z)=\\
=\diam Y-\diam X+r_1+r_2=\diam Z-\diam X,
\end{multline*}
therefore,
\begin{multline*}
2d_{GH}(X,Z)=\diam Z-\diam X=\\
=\big(\diam Z-\diam Y\big)+\big(\diam Y-\diam X\big)=\\
=2d_{GH}(Y,Z)+2d_{GH}(X,Y),
\end{multline*}
that is required.
\end{proof}

\begin{rk}
In Assertion~\ref{ass:hypo_exten_deux}, the metric segment $[X,Y]$ can be extended beyond $Y$ in different ways: as to $Z_r(y_0)$, so as to $Z_r(y_1)$, and to $Z_{r_1, r_2}(y_1,y_2)$ as well. Therefore, if say $X,Y\in\cM$, where any two points are connected by a shortest curve, than any shortest curve connecting $X$ and $Y$ can branch out at the point $Y$, similarly to the situation when a geodesic comes to the vertex of a flat cone with total angle at this vertex more than $2\pi$. In particular, this situation takes place for $X=\D_1$ and the standard ``radial'' shortest curve of the form $t\mapsto t\,Y$, $t\in[0,1]$.
\end{rk}

\begin{rk}
If we restrict ourselves by the space $\cM$ of compact metric spaces, then for any $X,Y\in\cM$ that satisfy the conditions of Assertion~\ref{ass:hypo_exten_deux}, and for any $r>0$, any shortest curve connecting $X$ and $Y$ can be extended beyond $Y$ to any space $Z$ from the intersection of the sphere $S_1(\D_1)$ of radius $1/2\diam Y+r/2$ centered at $\D_1$ with the sphere $S_2(Y)$ of radius $r/2$ centered at $Y$. In particular, this intersection contains all single-point extensions of the space $Y$ with the parameter $r$, and all its two-point extensions with parameters $r_1$, $r_2$, where $r_1+r_2=r$.
\end{rk}

The proof of the next Lemma can be obtained by a trivial modification of the proof of Lemma~\ref{lem:z_rr_curve}.

\begin{lem}\label{lem:z_rr_z_ss_curve}
Assume that $Y\in\cB$ and there exist $y_1,y_2\in Y$ such that $\diam Y=|y_1y_2|$. Let  $Z_{r_1,r_2}(y_1,y_2)$ be a two-point extension of the space $Y$ defined above. Then for any $0\le s_1\le r_1$ and $0\le s_2\le r_2$ the equality
$$
2d_{GH}\big(Z_{r_1,r_2}(y_1,y_2),Z_{s_1,s_2}(y_1,y_2)\big)=s_1-r_1+s_2-r_2
$$
holds.
\end{lem}

Lemma~\ref{lem:z_rr_z_ss_curve} implies the following result.

\begin{ass}\label{ass:monotonic_curve}
Let $Y$ and $Z_{r_1,r_2}(y_1,y_2)$ be such as in Lemma~$\ref{lem:z_rr_z_ss_curve}$, and $r(t)$ and $s(t)$ be non-negative strictly increasing continuous functions on $t\in I$, where $I$ is a finite or infinite interval. Then the curve $\g(t)=Z_{r(t),s(t)}(y_1,y_2)$ is shortest.
\end{ass}

\subsection{Examples of Non-Extendable Segments}
\noindent By means of Theorem~\ref{thm:ext_extens}, one can construct examples of non-extendable metric segments and non-extendable shortest curves.

\begin{examp}\label{ex:sim_sim}
Let $X=\l\D_n$ and $Y=\l \D_m$ be single-distance spaces, and $1<m<n$. Then the metric segment $[X,Y]$ can not be extended beyond $Y$.

Indeed, $2d_{GH}(X,Y)=\l=\diam X=\diam Y$, therefore the spaces $X$ and $Y$ are mutually extreme with respect to each other. Further, $X$ and $Y$ can be partitioned into $n$ and $m$ single-point subsets of zero diameter, respectively, therefore for $m<n$ we are under assumptions of Theorem~\ref{thm:ext_extens}.
\end{examp}

We need another simple corollary from Theorem~\ref{thm:dist-n-simplex-same-dim-alpha0}, see~\cite{GrigIT_Sympl}.

\begin{cor}\label{cor:alpha0}
Let $X$ be a bounded metric space, $1<m\le\#X$, and $\a_m(X)=0$. Then
$$
2d_{GH}(\l\D_m,X)=\max\bigl\{d_m(X),\,\l,\,\diam X-\l\bigr\}.
$$
\end{cor}

\begin{examp}\label{ex:two_points}
Now, let $X=\D_2$ be the single-distance space consisting of two points with non-zero distance $1$, and let $Y=[0,1]$ be the segment of the length $1$. Then $\a_2(Y)=0$, $d_2(Y)=1/2$, and hence $2d_{GH}(\D_2,Y)=1$ in accordance with Corollary~\ref{cor:alpha0}. As we have already seen in Corollary~\ref{cor:extendable}, the metric segment $[X,Y]$ can be extended beyond the simplex $X$. On the other hand, the spaces $X$ and $Y$ are mutually extreme. Represent the simplex $X$ as the union of its two points, and the segment $Y$ as the union of two segments $Y_1=[0,1/2]$ and $Y_2=[1/2,1]$. Theorem~\ref{thm:ext_extens} implies that $[X,Y]$ can not be extended beyond $Y$.
\end{examp}

\begin{examp}\label{ex:k_points}
Example~\ref{ex:two_points} can be generalized as follows. Let $X=\D_k$ be the unit simplex consisting of $k$ points, $1<k<\infty$, and $Y$ be a closed convex body of diameter $1$ with a smooth boundary in $(k-1)$-dimensional Euclidean space. Then $\a_k(Y)=0$, and the metric segment $[X,Y]$ can be extended beyond the simplex $X$. Due to Hadwiger's theorem~\cite{Hadw}, see also~\cite{Bolt}, the space $Y$ can be covered by $k$ subsets of diameters less than $1$, therefore $d_k(Y)<1$, and so, $2d_{GH}(\D_k,Y)=1$ in accordance with Corollary~\ref{cor:alpha0}. So, the spaces $X$ and $Y$ are mutually extreme. Represent the simplex $X$ as the union of its points, and cover $Y$ as above by $k$ subsets of diameters less than $1$. Then Theorem~\ref{thm:ext_extens} can be applied, and the segment $[X,Y]$ can not be extended beyond $Y$.
\end{examp}

\subsection{Two-Sides Infinite Non-Extendability}
In this Section we show that no one of metric segments in the space $\cB$ can be infinitely extended to both sides. We start with the following Lemma.

\begin{lem}\label{lem:out_ball}
There is no metric segments $[X_1,X_2]\ss\GH$ of the length $s$ in the interior of the ball of radius $s$ centered at the single-point space $\D_1$.
\end{lem}

\begin{proof}
Due to Assertion~\ref{ass:estim}, we have $2s\le\max\{\diam X_1,\diam X_2\}$, and since $2d_{GH}(\D_1,X_i)=\diam X_i$, then the distance from $\D_1$ to one of $X_i$ is more or equal than $s$. Lemma is proved.
\end{proof}

\begin{ass}\label{ass:inf_ext_2sides}
Let $Z\in\cB$ be an interior point of an $\e$-shortest curve $\g$. Then at least one of the ends of the curve $\g$ is contained in the ball of radius $R=\diam Z+\e$ centered at the single-point space $\D_1$.
\end{ass}

\begin{proof}
Let an $\e$-shortest curve $\g$ connect points $X_1$ and $X_2$, and both these points lie outside the ball of radius $R$ centered at $\D_1$. Then the segment of the curve $\g$ between $Z$ and $X_i$, $i=1,2$, contains a point $Y_i$ lying at the sphere $S$ of a radius $R'>R$ centered at $\D_1$. Therefore, $\diam Y_i=2R'$, $i=1,2$, and $2d_{GH}(Y_i,Z)\ge\diam Y_i-\diam Z=2R'-\diam Z$, $i=1,2$. Further, the segment $\g$ between $Y_1$ and $Y_2$ is also an $\e$-shortest curve. Therefore, due to Lemma~\ref{lem:eps-geod},
$$
d_{GH}(Y_1,Y_2)\ge d_{GH}(Y_1,Z)+d_{GH}(Z,Y_2)-\e\ge2R'-\diam Z-\e.
$$
On the other hand, $d_{GH}(Y_1,Y_2)\le\frac12\max\{\diam Y_1,\diam Y_2\}=R'$, and hence, $R<R'\le\diam Z+\e$, a contradiction.
\end{proof}

\begin{cor}
No metric segment can be infinitely extended beyond both its endpoints.
\end{cor}

\begin{proof}
Assume to the contrary that for an arbitrary $\e>0$ and each $i=1,\,2$ there exists an extension of some metric segment $[X_1,X_2]$ beyond its end $X_i$ to some $Y_i$ such that $d_{GH}(X_i,Y_i)>\max\{\diam X_1,\diam X_2\}+\e$. Due to extension definition, $Y_i\in\cB$, $i=1,2$. In accordance to Corollary~\ref{cor:intrinsic}, there exist $(\e/3)$-shortest curves connecting $Y_1$ and $X_1$,  $X_1$ and $X_2$, and $X_2$ and $Y_2$, respectively. By Lemma~\ref{lem:union_eps_geo}, the union of these three curves is an $\e$-shortest curve between $Y_1$ and $Y_2$, and it contains $X_1$ and $X_2$ in its (relative) interior. By Assertion~\ref{ass:inf_ext_2sides}, one of the end-points of this curve, i.e., one of the points $Y_i$, say $Y_1$, is contained in the ball of radius $\max\{\diam X_1,\diam X_2\}+\e$ centered at $\D_1$, i.e., $\diam Y_1\le2\max\{\diam X_1,\diam X_2\}+2\e$. By Assertion~\ref{ass:estim}, we have
$$
d_{GH}(Y_1,X_1)\le\frac12\max\{\diam Y_1,\diam X_1\}\le\max\{\diam X_1,\diam X_2\}+\e,
$$
a contradiction.
\end{proof}

\markright{References}


\begin{thebibliography}{99}

\bibitem{BurBurIva} D.~Burago, Yu.~Burago, and S.~Ivanov, \emph{A Course in Metric Geometry}. Graduate Studies in Mathematics, vol.~33. A.M.S., Providence, RI, 2001.

\bibitem{Herron} D.\,A.~Herron, ``Gromov--Hausdorff Distance for Pointed Metric Spaces'', J. Anal.,  {\bf 24} (1), pp.~1--38 (2016).

\bibitem{IvaIliadisTuz} A.\,O.~Ivanov, S.~Iliadis, A.\,A.~Tuzhilin, ``Realizations of Gromov-Hausdorff Distance'', ArXiv e-prints, {\tt arXiv:1603.08850} (2016).

\bibitem{Ghanaat} P.~Ghanaat,  ``Gromov-Hausdorff distance and applications'', In: \emph{Summer school ``Metric Geometry''}, Les Diablerets, August 25--30, 2013, \url{https://math.cuso.ch/fileadmin/math/document/gromov-hausdorff.pdf}

\bibitem{Memoli2018} S.~Chowdhury, F.~Memoli,  ``Explicit Geodesics in Gromov-Hausdorff Space'', ArXiv e-prints, {\tt arXiv:1603.02385} (2018).

\bibitem{IvaNikolaevaTuz}
A.\,O.~Ivanov, N.\,K.~Nikolaeva, A.\,A.~Tuzhilin, ``The Gromov--Hausdorff Metric on the Space of Compact Metric Spaces is Strictly Intrinsic'', ArXiv e-prints, {\tt arXiv:1504.03830} (2015); Mathematical Notes,  {\bf 100} (6), pp.~171--173 (2016).

\bibitem{Memoli} S.~Chowdhury, F.~Memoli,  ''Constructing Geodesics on the Space of Compact Metric Spaces'', ArXiv e-prints, {\tt arXiv:1603.02385} (2016).

\bibitem{GrigIT_Sympl}
D.\,S.~Grigorjev, A.\,O.~Ivanov, A.\,A.~Tuzhilin, ``Gromov--Hausdorff Distance to Simplexes'', ArXiv e-prints, {\tt arXiv:1906.09644} (2019); Chebyshevskii Sbornik {\bf 20} (2), pp.~100--114 (2019).

\bibitem{Hadw}
H.~Hadwiger,  ``\"Uberdeckung einer Menge durch Mengen kleineren Durchmessers'', Commentarii Mathematici Helvetici, {\bf 18} (1), pp.~73--75 (1945).

\bibitem{Bolt}
V.~Boltjansky, I.~Gohberg,
\emph{Results and Problems in Combinatorial Geometry}, Cambridge University Press (1985).

\end{thebibliography}
\end{document}